\documentclass[10pt]{amsart}
\usepackage{amssymb}
\usepackage{epsfig}
\usepackage{url}
\usepackage{setspace}
\usepackage{pdflscape}
\usepackage{tikz}

\usetikzlibrary{automata,positioning,arrows}
\tikzset{->,
>=stealth,
node distance=3cm}

\theoremstyle{plain}

\newtheorem{thm}{Theorem}[section]
\newtheorem{cor}[thm]{Corollary}
\newtheorem{lem}[thm]{Lemma}

\newtheorem{ques}[thm]{Question}
\newtheorem{prob}[thm]{Problem}

\def\bbb{\mathbb}
\def\op{\operatorname}
\renewcommand{\phi}{\varphi}

\newcommand{\N}{\bbb{N}}
\newcommand{\Z}{\bbb{Z}}
\newcommand{\Q}{\bbb{Q}}

\newcommand{\C}{\bbb{C}}

\makeatletter
\let\@@pmod\pmod
\DeclareRobustCommand{\pmod}{\@ifstar\@pmods\@@pmod}
\def\@pmods#1{\mkern4mu({\operator@font mod}\mkern 6mu#1)}
\makeatother

\begin{document}

\title[Solutions of certain meta-Fibonacci recurrences]{Solutions of certain meta-Fibonacci recurrences}
\author{Bartosz Sobolewski and Maciej Ulas}

\keywords{meta-Fibonacci sequence, recurrence sequence, automatic sequence, binary partitions}
\subjclass[2010]{}
\thanks{The research of the authors is supported by the grant of the National Science Centre (NCN), Poland, no. UMO-2019/34/E/ST1/00094}

\begin{abstract}
In this note we investigate the solutions of certain meta-Fibonacci recurrences of the form $f(n)=f(n-f(n-1))+f(n-2)$ for various sets of initial conditions. In the case when $f(n)=1$ for $n\leq 1$, we prove that the resulting integer sequence is closely related to the function counting binary partitions of a certain type.
\end{abstract}

\maketitle

\section{Introduction}\label{sec1}
The Fibonacci sequence $(F_{n})_{n\in\N}$, defined as $F_{0}=0, F_{1}=1$ and $F_{n}=F_{n-1}+F_{n-2}, n\geq 2$, has many diverse applications and is still investigated from different points of view. In the literature, there are many various generalizations of this sequence, many of which can be seen as subclasses of the class of so-called {\it meta-Fibonacci sequences}. What is a meta-Fibonacci sequence? In the most general terms, it is a solution of a recurrence relation of the form
\begin{equation}\label{genmeta}
f(n)=f(p(n))+f(q(n)),
\end{equation}
where $p(n), q(n)$ are certain expressions involving $n, f(n-1), \ldots, f(n-k)$ for some $k$. To get a sequence, we also need to specify certain initial conditions $f(1),\ldots, f(i)$ (where $i$ depends on $p, q$). We note that the class of sequences defined in this way is larger than the one introduced by Conolly in \cite{Con}.

The class of meta-Fibonacci sequences is very broad and (trivially) contains the Fibonacci sequence (take $p(n)=n-1, q(n)=n-2$), the sequence of powers of $2$ (take $p(n)=q(n)=n-1$ with $f(1)=1$), but also many exotic sequences. One striking example is the Hofstadter sequence, which is defined in the following way: $Q(1)=Q(2)=1$ and for $n\geq 3$ we have
$$
Q(n)=Q(n-Q(n-1))+Q(n-Q(n-2)).
$$
This sequence first appeared in the book \cite{Hof} and since its invention has attracted much attention.

In contrast to usual linear recurrence sequences, for certain initial conditions, the corresponding meta-Fibonacci sequence may not be well defined for all $n \in\N_+$. Indeed, if for some $n$ we have $p(n)\leq 0$ or $q(n)\leq 0$ then the term with index $n$ can not be computed. Then we say that the sequence {\it dies}. We get a typical example for $p(n)=n-f(n-1), q(n)=n-f(n-2)$ with initial conditions $f(1)=2, f(2)=1$. Then $f(3)=3, f(4)=5$ but we can not compute the value of $f(5)$  because $5-f(4)=0$ and $f(0)$ is not defined. The question if a given meta-Fibonacci sequence dies is in general hard to answer. For example, it is still not known whether the Hofstadter sequence $Q$ will live forever or die at some $n$. A common belief is that the sequence is well defined for each $n \in \N_+$, which is supported by calculations for $n\leq10^{10}$. In order to eliminate the problem with dying sequences, one can use the convention that the values at negative indices are fixed. For example, Ruskey \cite{Rus} discovered a solution to Hofstadter’s $Q$-recurrence (with appropriate initial conditions and the convention that $Q(n)=0$ for $n<0$) that includes the Fibonacci along indexes lying in certain arithmetic progression.

An interesting variant of the Hofstadter sequence with perfectly predictable behavior was proposed by Tanny \cite{Tan}. The sequence starts with $T(0)=T(1)=T(2)=1$ and for $n\geq 3$ we have
$$ T(n)=T(n-1-T(n-1))+T(n-2-T(n-2)). $$
Tanny proved that this sequence is {\it slow}, well-defined for all $n\in\N$ and each positive integer appears as a value of $T(n)$. We say that a sequence $(f(n))_{n\in\N}$ is slow if for each $n\in\N$ we have $f(n+1)-f(n)\in\{0,1\}$. Some other constructions of slow meta-Fibonacci sequences were presented in the literature (see for example \cite{IL}).

Browsing the literature concerning meta-Fibonacci recurrences, we have observed that in every case studied so far both expressions $p(n), q(n)$ used in the definition of the recurrence (\ref{genmeta}) contain at least one term of the form $f(n-k)$ for some $k$. Moreover, in each case when the sequence is well-understood, there is some $a\in\N_{+}$ such that for each $b\in\{0,\ldots,a-1\}$, the subsequence $(f(an+b))_{n\in\N}$ is bounded from above by a linear function in $n$ or is governed by a specific linear recurrence sequence with exponential growth (with algebraic base).

This motivated us to ask the following question. Consider the recurrence (\ref{genmeta}) with $p(n)=n-f(n-u)$ and $q(n)=n-v$, where $u, v$ are fixed positive integers, namely
\begin{equation}\label{genmeta1}
f(n) = f(n- f(n-u))+f(n-v).
\end{equation}

\begin{ques} \label{main_question}
Can we expect the existence of a solution such that at least one of its subsequences along some arithmetic progression is not a linear recurrence sequence, neither is slow, nor unpredictable (chaotic)?
\end{ques}

%In general, to avoid the situation when the resulting sequence dies, we may define $f(n) =a$ for $n\leq 0$, where $a \geq 1$ is a constant. The case $u=v=1$ is then very easy to describe and one can show by induction that $f(n)=a(n+1)$ for all $n\in\N$.
%If $u>1$ and $v=1$ then under the same assumption on initial conditions as in the case $u=v=1$, we get that $f(n)=n+3$ for $a=1$ and $f(n)=a(n+1)$ for $n>2$ (note that this solution is independent of $u$).
%In the case $u > v = 1$ the sequence

%In general, to avoid the situation when the resulting sequence dies, we may define $f(n) =c$ for $n\leq 0$, where $c \geq 0$ is a constant. If $v=1$, then it is easy to observe that the resulting sequence does not exhibit the desired behavior as in Question \ref{main_question}. Indeed, we then have $f(n) \geq f(n-1)+1$. If $f(n-u)$

In general, any initial condition can be considered $f(0),f(1), \ldots, f(u-1) \in \N_{+}$. The assumption that these values are positive ensures that $n-f(n-u)<n$, which is necessary to compute $f(n)$ recursively. To avoid the situation in which the resulting sequence dies, we assign a fixed value $c \in \N$ to terms with indices $n < 0$.

The case $v=1$ is rather uninteresting in view of Question \ref{main_question}. Indeed, if $c \geq 1$, then for any $u \geq 1$ the resulting sequence $(f(n))_{n\in\N}$ is eventually an arithmetic progression. To see this, note that the recurrence relation \eqref{genmeta1} implies that this sequence is strictly increasing. If $f(n-u) \leq n$ for all $n \geq u$, then we must have $f(n) = f(n-1)+1$ for sufficiently large $n$. On the other hand, if $f(n_0-u) > n_0$ for some $n_0 \in \N$, then it is easy to inductively prove that $f(n) = f(n-1) + c$ for all $n\geq n_0$.
The description for $c=0$ is more complicated; however, experimental computations strongly indicate that there exists some $d \geq 1$ (depending on the initial conditions $f(0),f(1), \ldots, f(u-1)$) such that for $e=0,1,\ldots,d-1$ the subsequences $(f(dn+e))_{n\in\N}$ are eventually arithmetic progressions.

Based on the above discussion, we investigate the problem stated for the next ``smallest'' case, namely $u=1, v=2$. More precisely, let $a, b\in\N_+$ and consider the sequence ${\bf h}_{a,b}=(h_{a,b}(n))_{n\in\N}$, where
$$
h_{a,b}(n)=\begin{cases}\begin{array}{lcc}
                          a &  & n\leq 0, \\
                          b &  & n=1,\\
                          h_{a,b}(n-h_{a,b}(n-1))+h_{a,b}(n-2) & & n>1.
                        \end{array}
\end{cases}
$$

We also define a related sequence ${\bf g}_{a,b}=(g_{a,b}(n))_{n\in\N}$, where
$$
g_{a,b}(n)=\begin{cases}\begin{array}{lcc}
                          0 &  & n<0, \\
                          a &  & n=0, \\
                          b &  & n=1,\\
                          g_{a,b}(n-g_{a,b}(n-1))+g_{a,b}(n-2) & & n>1.
                        \end{array}
\end{cases}
$$
As we see the sequence ${\bf g}_{a,b}$ is governed by the same recurrence as ${\bf h}_{a,b}$ but with different initial conditions.
There are striking differences in the behavior between the sequences ${\bf h}_{a,b}$ and ${\bf g}_{a,b}$ for various values of $a, b$. In particular, it turns out that for $a=1$ (and any $b \geq 1$) the solution ${\bf h}_{a,b}$ satisfies the requirements listed in Question \ref{main_question}.

Let us describe the contents of the paper in more detail. Section \ref{sec2} provides the necessary background information on automatic and regular sequences.

In Sections \ref{sec3}--\ref{sec5} we offer a detailed description of the sequence ${\bf h}_{a,b}$.

Section \ref{sec3} is devoted to the study of the particularly interesting case $a=b=1$. We obtain $h_{1,1}(2n+1)=n+1$ but the solution along even indices can be presented as a sum of functions counting certain partitions into powers of two (a rather surprising property, given the definition of ${\bf h}_{1,1}$). In particular, the sequence ${\bf h}_{1,1}$ does not satisfy a linear recurrence equation and its ordinary generating function is transcendental over the rational function field. Moreover, we also prove that the sequence obtained by reduction modulo 2 of the sequence ${\bf h}_{1,1}$ is 2-automatic.

In Section \ref{sec4} we analyze the case $a=1, b>1$. The sequence $h_{1,b}$ displays similar behavior as $h_{1,1}$. The main result is that $h_{1,b}(2n)=n+1$ and that $h_{1,b}(2n+1)$ is a linear combination of the sequence $(bin(n))_{n\in\N}$ (which counts partitions into power of two) and the sequence ${\bf h}_{1,1}$. Section \ref{sec5} deals with the remaining case $a > 1$. We prove for all $b \in \N_+$ the subsequences $h_{a,b}(2n)$ and  $h_{a,b}(2n+1)$ are eventually arithmetic progressions.

In Section \ref{sec6} we focus on the sequence ${\bf g}_{a,b}$. We prove that for any initial conditions $a, b\in\N_{+}$, the behavior of ${\bf g}_{a,b}$ is completely predictable. More precisely, depending on the parity of $a, b$ the subsequences $g_{a,b}(2n)$ and $g_{a,b}(2n+1)$ are eventually: constant, linear or satisfy linear recurrence equation of order bounded by $\op{max}\{a, b\}$. Finally, in the last section we formulate some questions and problems which may stimulate further research.

\section{Preliminaries on automatic and regular sequences}\label{sec2}
In this section, we recall the definition and some properties of automatic and regular sequences, which will be useful later in the paper. These sequences appear naturally in various branches of mathematics, such as number theory, theoretical computer science, symbolic dynamics, etc. For a detailed treatment of this topic, we refer the reader to the monograph by Allouche and Shallit \cite{AS03a} and their papers \cite{AS92, AS03b}.

Let $k \geq 2$ be a fixed integer. A sequence $\mathbf{a} = (a_n)_{n\in\N}$ is called {\it $k$-automatic} if its {\it $k$-kernel}, namely
$$ K_k(\mathbf{a}) = \{(a_{k^jn+i})_{n\in\N}: j \in \N, 0 \leq i < k^j \},
$$
is a finite set. Equivalently, $\mathbf{a}$ is $k$-automatic if there exists a deterministic finite automaton with output (DFAO) which reads the canonical base-$k$ representation of $n$ and outputs $a_n$.

One of the most famous automatic sequences is the Prouhet-Thue-Morse sequence $\mathbf{t}=(t_{n})_{n\in\N}$ (the PTM sequence for short). It is 2-automatic sequence defined in the following way: $t_{0}=1$ and
$$
t_{2n}=t_{n},\quad t_{2n+1}=-t_{n},\quad n\in \N.
$$
From this description it is clear that $K_k(\mathbf{t}) =\{\mathbf{t}, -\mathbf{t}\}$. Equivalently, the sequence is generated by the DFAO in Figure \ref{PTM_DFAO}. To generate $t_n$, one moves between the states (symbolized by nodes) according to subsequent digits in the binary representation of $n$. After the last digit has been read, the DFAO returns the output corresponding to the final state.
\begin{figure}[h!]
\centering
\begin{tikzpicture}
\node[state, initial](q0){$1$};
\node[state, right of=q0](q1){$-1$};

\draw
 (q0) edge[loop above] node{0} (q0)
 (q0) edge[bend left, above] node{1} (q1)
 (q1) edge[bend left, below] node{1} (q0)
 (q1) edge[loop above] node{0} (q1);
\end{tikzpicture}
\caption{A DFAO generating the PTM sequence.}
\label{PTM_DFAO}
\end{figure}
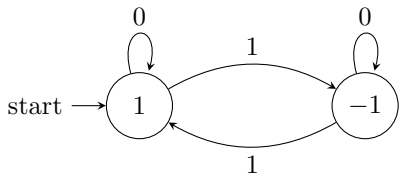

A famous result of Christol \cite{Chr79} relates the automaticity of sequences and the algebraicity of generating functions over finite fields.

\begin{thm}
Let $p$ be a prime number and $\mathbb{F}_q$ a finite field of order $q=p^k$, where $k\in\N_+$. Let $\mathbf{a} =(a_n)_{n\in\N}$ be a sequence with values in $\mathbb{F}_q$  and $A(X) = \sum_{n=0}^{\infty} a_n X^n$ be its generating function. Then $\mathbf{a}$ is $p$-automatic if and only if $A(X)$ is algebraic over the field of rational functions $\mathbb{F}_q(X)$.
\end{thm}

A natural generalization of $k$-automatic sequences are so-called {\it $k$-regular sequences}. More precisely, assume that a sequence $\mathbf{a} = (a_n)_{n\in\N}$ takes values in a $\Z$-module $R$. Then $\mathbf{a}$ is $k$-regular if there exist finitely many sequences $\mathbf{a}_{i} = (a_i(n))_{n\in\N}$ with values in $R$ such that each sequence in $K_k(\mathbf{a})$ is a $\Z$-linear combination of the $\mathbf{a}_{i}$. In other words, the $\Z$-submodule generated by $K_k(\mathbf{a})$ is finitely generated.

It is rather easy to check that $k$-automatic sequences are precisely $k$-regular sequences taking only finitely many values. The following result \cite[Theorem 2.10]{AS92} asserts that the growth of a $k$-regular sequence is at most polynomial.

\begin{thm}
Let $(a_n)_{n\in\N}$ be a $k$-regular sequence with values in $\mathbb{C}$. Then there exists a constant $C$ such that $a_n = O(n^C)$.
\end{thm}

\section{Analysis of the sequence ${\bf h}_{1,1}$}\label{sec3}

%In this section we analyze the behaviour of the sequence $(h_{a,b}(n))_{n\in\Z}$ with $a=b$.
 %Let $a\in\N$ be fixed and consider the sequence $(h_{a}(n))_{n\in\Z}$ defined in the %following way:
%$$
%h_{a}(n)=\begin{cases}\begin{array}{lcc}
%                          a &  & n<2 \\
%                          h_{a}(n-h_{a}(n-1))+h_{a}(n-2) & & n\geq 2
%                        \end{array}
%\end{cases}.
%$$
%Due to the condition $h_{a,a}(n)=a$ for $n<2$ the sequence is well defined for $n\in\Z$. First we consider the case $a\geq 2$ and show the following
%
%\begin{thm}
%For $n\in\N$ we have that $h_{a}(n)=a\left\lfloor\frac{n+2}{2}\right\rfloor$.
%\end{thm}
%\begin{proof}
%Just check.
%\end{proof}
%
%The case $a=1$ is completely different and presents interesting and non-trivial behaviour.

In this section, we analyze the behavior of the sequence $(h_{1,1}(n))_{n\in\N}$.
We write $h(n)=h_{1,1}(n)$ and in the sequel, we consider the sequence
${\bf h}=(h(n))_{n\in\N}$. The terms $h(n)$ for $n=0, 1, 2, \ldots, 25$ are as follows:
$$
1, 1, 2, 2, 4, 3, 6, 4, 10, 5, 13, 6, 19, 7, 23, 8, 33, 9, 38, 10, 51, 11, 57, 12, 76, 13, \ldots .
$$

To begin, we prove that the terms $h(2n)$ and $h(2n+1)$ satisfy simple formulas, eliminating the need to calculate the nested term $h(n-h(n-1))$.

\begin{thm}\label{rech}
The sequence $\mathbf{h}$ satisfies the following recurrence relations: $h(0) = h(1) = 1$ and for all $n\in\N$ we have
\begin{align*}
h(2n+1)&=n+1, \\
h(2n+2)&=h(2n)+h(n+1)
\end{align*}
\end{thm}
\begin{proof}
We first use induction on $n$ to simultaneously prove that $h(2n)\geq 2n$ and $h(2n+1)=n+1$. It is clear that this claim is true for $n=0, 1, 2$. Let us suppose that this holds for all numbers $<2n$. Then we have $$h(2n)=h(2n-h(2n-1))+h(2n-2)=h(n)+h(2n-2).$$
We also see that the result will follow as soon as our claim is proved.

If $2|n$, then from our assumption $h(n)\geq n$, and thus for $n\geq 2$ we get
$$h(2n)\geq n+2(n-1)=3n-2\geq 2n.$$
If $2\nmid n$, then for $n\geq 2$ we have
$$h(2n)\geq \frac{n+1}{2}+2n-2\geq 2n,$$
and we obtain the first part of our claim. To obtain the second part, we note that
$$h(2n+1)=h(2n+1-h(2n))+h(2n-1)=1+n$$
because $2n+1-h(2n)\leq 1$. The result follows.
\end{proof}

Although very simple, the above result has several nice applications. We start with the following.

\begin{cor} \label{rech_corollary}
For all $n \in \N$ we have
$$ h(2n) = \sum_{i=0}^n h(i).  $$
\end{cor}

Furthermore, we obtain interesting identities involving the terms of the Prouhet-Thue-Morse sequence. For $k \in \N$ and $n \geq 2^{k+1}-2$ let us define
\begin{equation} \label{F_kn}
    F(k,n) = \sum_{i=0}^{2^k-1} t_i h(n-2i).
\end{equation}

\begin{thm} \label{thm_ptm}
For all $k\in \N$ and $n \geq 2^{k+1}-2$ we have the following relations:
\begin{align*}
    F(k,2n) &= \begin{cases}
                  h(2n)  &\text{if } k=0,\\
                  F(k-1,n) &\text{if } k\geq 1,
            \end{cases} \\
   F(k,2n+1) &= \begin{cases}
   n+1  &\text{if } k=0, \\
   1    &\text{if } k=1, \\
   0    &\text{if } k \geq 2.
   \end{cases}
\end{align*}
In particular, for any $n \in \N_{+}$, we get the identity
\begin{equation} \label{ptm_identity}
  \sum_{i=0}^{2^{k}-1}t_{i}h(2^{k+1}n+2^{k}-2i)=n+1.
\end{equation}
\end{thm}
\begin{proof}
The formula $F(0,2n)=h(2n)$ follows from the definition. Now, fix $n\in\N$ and write $f(n,i)=h(n-2i)$. Theorem \ref{rech} implies that we have
$$ f(2n,2i) = f(2n,2i+1) + f(n,i).   $$
Therefore, for $k \geq 1$ we get
\begin{align*}
    F(k,2n) &= \sum_{i=0}^{2^{k}-1}t_{i}f(2n,i) =
    \sum_{i=0}^{2^{k-1}-1} \left[ t_{2i} f(2n,2i) + t_{2i+1} f(2n,2i+1) \right] \\
    &= \sum_{i=0}^{2^{k-1}-1} t_i \left[ f(2n,2i) - f(2n,2i+1) \right]  = \sum_{i=0}^{2^{k-1}-1}t_{i}f(n,i) = F(k-1,n).
\end{align*}

Moving on to $F(k,2n+1)$, we have $F(0,2n+1) = h(2n+1)=n+1$ by Theorem \ref{rech}. Similarly,
$$ F(1,2n+1) = h(2n+1) - h(2n-1) = n+1 - n = 1.  $$
For $k \geq 2$ we have
$$ F(k,2n+1) = \sum_{i=0}^{2^{k}-1} t_i h(2(n-i)+1) = \sum_{i=0}^{2^{k}-1} t_i(n-i+1) = 0, $$
where the last equality follows from the following identity \cite{Cer}, valid for all $m < k$:
\begin{equation} \label{eq:ptm_polynomial}
    \sum_{i=0}^{2^{k}-1}t_{i}i^{m}=0.
\end{equation}

Finally, \eqref{ptm_identity} is obtained by applying $k$ times the relation $ F(k,2n) = F(k-1,n)$.
\end{proof}

As an immediate consequence, we get the following corollary.

\begin{cor}\label{cor1}
There is no polynomial $P\in\Z[x]$ such that $h(2n)=P(n)$ for infinitely many $n\in\N$.
\end{cor}
\begin{proof}
As a consequence of \eqref{eq:ptm_polynomial}, for each polynomial $P\in\Z[x]$ of degree $<k$ we get that $\sum_{i=0}^{2^{k}-1}t_{i}P(x+i)=0$.

Let us suppose that there is a polynomial $P\in\Z[x]$ of degree $d\in\N_{+}$ satisfying $h(2n)=P(n)$ for $n\in\N$. However, from \eqref{ptm_identity} for fixed $n\in\N$ and $k>d$ we would have
$$
n+1=\sum_{i=0}^{2^{k-1}-1}t_{i}Q(i)=0,
$$
where $Q\in\Z[x]$ satisfies $Q(x)=P(2^{k+1}n+2^{k}-2x)$. We get a contradiction.
\end{proof}

In the next result, we establish a functional equation for the ordinary generating function of the sequence ${\bf h}$, namely
$$H(x)=\sum_{n=0}^{\infty}h(n)x^{n}.$$

\begin{thm}\label{genfunh}
The series $H$ satisfies a Mahler-type functional equation of the form
$$
\frac{1}{1-x^2}H(x^2)-H(x)+\frac{x}{(1-x^2)^{2}}=0.
$$
\end{thm}
\begin{proof}
We have the identities
\begin{align*}
\frac{1}{2}(H(x)+H(-x))&=\sum_{n=0}^{\infty}h(2n)x^{2n}=\sum_{n=0}^{\infty}\left(\sum_{i=0}^{n}h(i)\right)x^{2n}=\frac{1}{1-x^{2}}H(x^2),\\
\frac{1}{2}(H(x)-H(-x))&=\sum_{n=0}^{\infty}h(2n+1)x^{2n+1}=\sum_{n=0}^{\infty}(n+1)x^{2n+1}=\frac{x}{(1-x^{2})^{2}}.
\end{align*}
If we add side-by-side of the above equalities, we obtain the statement.
\end{proof}

As an immediate consequence, we get the following corollary.

\begin{cor}\label{transH}
The generating function $H(x)$ of the sequence ${\bf h}$ is transcendental over $\Q(x)$. In particular, neither the sequence ${\bf h}$ nor its subsequence $(h(2n))_{n\in\N}$ are linear recurrence sequences.
\end{cor}
 \begin{proof}
We treat the series $H$ as a element of $\Z[[x]]\subset \C[[x]]$. Let us recall the Nishioka theorem \cite{Ni1} (see also \cite{Ni2}) which says that if $f\in \C[[x]]$ satisfies a functional equation of the form $f(x^{d})=\phi(x,f(x))$, where $\phi\in \C(x,y)$, then $f$ is either rational or transcendental. In our case, $d=2$ and
$$
\phi(x,y)=(1-x^2)y-\frac{x}{1-x^2}.
$$
Thus, to get the result, it is enough to prove that the series $H$ is not rational. To see this, it suffices to consider the reduction modulo 2 of $H$, i.e., the series $G(x)=H(x)\pmod*{2}$. The series $G$ is an element of $\mathbb{F}_{2}[[x]]$ and satisfies the algebraic equation
\begin{equation} \label{G_algebraic_equation}
(1-x^2)G(x)^2+(1-x^2)^2G(x)+x=0
\end{equation}
(this is consequence a of the congruence $G(x^2)\equiv G(x)^2\pmod*{2}$). However, it is easy to check that the polynomial $P(x,y)=(1-x^2)y^2+(1-x^2)^2y+x$ is irreducible in $\mathbb{F}_{2}[x,y]$. We thus see that the series $G$ is not rational over $\mathbb{F}_{2}(x)$ and thus $H$ can not be rational. The Nishioka theorem implies its transcendentality over $\C(x)$ and our theorem is proved.
\end{proof}

Before we prove our next result, we recall some basic facts about the binary partition function.
Let us write
$$
B(x)=\prod_{n=0}^{\infty}\frac{1}{1-x^{2^{n}}}=\sum_{n=0}^{\infty}bin(n)x^{n}.
$$
The number $bin(n)$ counts the representations of an integer $n$ in the form $n=\sum_{i=1}^{k}2^{r_{i}}$, where $r_{i}\in\N$ for $i=1 \ldots, k$ and $r_{1}\leq r_{2}\leq \ldots\leq r_{k}$. In other words, $bin(n)$ counts the number of partitions of $n$ with parts in the set $\{2^{i}:\;i\in\N\}$. We can deduce that $B$ satisfies the functional equation $(1-x)B(x)=B(x^2)$ and the sequence $(bin(n))_{n\in\N}$ can be defined recursively as follows:
$$
bin(0)=1, \quad bin(2n)=bin(2n-1)+bin(n),\quad bin(2n+1)=bin(2n).
$$
This sequence was introduced by Euler but it seems that the first non-trivial results were proved by Churchhouse in \cite{Ch}.
The most important fact for our purposes is that the value of $bin(n)$ grows exponentially. More precisely, Mahler proved that $\log_{2} bin(n)\sim \frac{1}{2}(\log_{2} n)^{2}$. For a proof of this fact, see \cite{Mah}.

On a related note, $B$ is the multiplicative inverse of the generating function $T$ for the PTM sequence:
$$T(x)=\sum_{n=0}^{\infty} t_n x^n = \prod_{n=0}^{\infty}(1-x^{2^{n}}) = \frac{1}{B(x)}.$$
 We are ready to prove the following

\begin{thm} \label{nonregular}
The subsequence $(h(2n))_{n\in\N}$ and thus the whole sequence ${\bf h}$ is not $k$-regular for any $k\in\N_{\geq 2}$.
\end{thm}
\begin{proof}
The functional equation for $H$ can be rewritten as
$$
H(x)=\frac{1}{1-x^2}H(x^2)+\frac{x}{(1-x^2)^{2}}.
$$
Iterating this equation by induction on $k$ one can easily prove that
$$
H(x)=\left(\prod_{i=1}^{k}\frac{1}{1-x^{2^{i}}}\right)H(x^{2^{k}})+\sum_{i=1}^{k}\frac{x^{2^{i-1}}}{\prod_{j=1}^{i}(1-x^{2^{j}})}\cdot \frac{1}{1-x^{2^{i}}}.
$$
Thus, taking the limit $k\rightarrow +\infty$, we get the equality $H(x)=\overline{H}(x)+\sum_{i=1}^{\infty}\overline{H}_{i}(x)$, where
\begin{align*}
    \overline{H}(x)&=\prod_{i=1}^{\infty}\frac{1}{1-x^{2^{i}}}=B(x^2), \\
    \overline{H}_{i}(x)&=\frac{x^{2^{i-1}}}{\prod_{j=1}^{i}(1-x^{2^{j}})}\cdot \frac{1}{1-x^{2^{i}}}.
\end{align*}

Let us note that $n$-th coefficient in the power series expansion of $\overline{H}$ is equal to 0 if $n$ is odd and is just $bin(n/2)$ if the number $n$ is even. In particular, the growth of the coefficients is exponential.

On the other hand, it is clear that coefficients of the power series expansion of the second part in the above sum are non-negative. More precisely, write $\overline{H}_{i}(x)=\sum_{n=0}^{\infty}c_{i}(n)x^{i}$, where $c_{i}(n)$ is the number of partitions of $n-2^{i-1}$ into elements of the set $\{2,\ldots, 2^{i}\}$, where the last element $2^{i}$ can take two colors. Therefore, if we write
$$
\sum_{k=1}^{\infty}\overline{H}_{i}(x)=\sum_{n=0}^{\infty}d(n)x^{n},
$$
then $d(n)=\sum_{i=1}^{\lfloor\log_{2} n\rfloor+1}c_{i}(n-2^{i-1})$. Therefore, we have $h(2n)=bin(n)+d(2n)\geq bin(n)$ and the sequence $(h(2n))_{n\in\N}$ has exponential growth and can not be $k$-regular for any $k$. It is clear that the same is true for the whole sequence ${\bf h}$.
\end{proof}

To conclude this section, we consider the sequence $(h(n))_{n\in\N}$ reduced modulo $2$. Let us write $r(n)=h(n)\pmod*{2}$. As a consequence of the functional equation for $H$ we immediately deduce that this sequence is $2$-automatic.
\begin{thm}
The sequence $(r(n))_{n\in\N}$ is 2-automatic.
\end{thm}
\begin{proof}
The generating function (over $\mathbb{F}_{2}(x))$ of the sequence $(r(n))_{n\in\N}$ is just $G(x)=H(x)\pmod*{2}$ and satisfies the algebraic equation \eqref{G_algebraic_equation} over $\mathbb{F}_{2}(x)$. From Christol's theorem, we get the result.
\end{proof}

By following the proof of Christol's theorem, one can derive from equation \eqref{G_algebraic_equation} a DFAO generating the sequence $(r(n))_{n\in\N}$. We are going to achieve the same goal through a more direct approach. Since $r(2n+1) \equiv n+1 \pmod{2}$, we are going to focus on the more interesting subsequence $(r(2n))_{n\in\N}$.

\begin{thm}
We have $(r(8n+4))_{n\in\N} = \overline{0,1,1,0}$ and $(r(16n+8))_{n\in\N} = \overline{0,0,1,1}$. Moreover, for all $n\in\N$ we have the following relations:
\begin{align*}
r(8n+2) &= 1-r(8n), \\
r(8n+6) &= r(8n+4), \\
r(32n) &= r(16n), \\
r(32n+16) &= 1-r(16n+8).
\end{align*}
\end{thm}
\begin{proof}
We are going to apply the relations
\begin{align}
r(2n+1) &\equiv n+1 \pmod{2}, \label{rec_mod2_1}  \\
r(2n+2) &\equiv r(2n) + r(n+1) \pmod{2}, \label{rec_mod2_2}
\end{align}
which follow directly from Theorem \ref{rech}.

We obtain the congruences
\begin{align*}
r(8n+2) &\equiv r(8n) + r(4n+1) \equiv r(8n) + 2n+1 \equiv r(8n) +1 \pmod{2}, \\
r(8n+6) &\equiv r(8n+4) + r(4n+3) \equiv r(8n+4) + 2n+2 \equiv r(8n+4) \pmod{2},
\end{align*}
which immediately give the first two claimed equalities.

We move on to the periodicity of $(r(8n+4))_{n\in\N}$ and $(r(16n+8))_{n\in\N}$. Applying \eqref{rec_mod2_2} repeatedly, we get for $n \in \N$ and $j\in\{0, 1, 2, 3\}$, the congruence
$$ r(32(n+1)+8j+4) \equiv r(32n+8j+4) + \sum_{i=0}^{15} r(16(n+1)+4j+2 -i) \pmod{2}.$$
Using the notation $F(k,n)$ as in \eqref{F_kn}, by Theorem \ref{thm_ptm} we get
\begin{align*}
    \sum_{i=0}^{15} r(16(n+1)+4j+2 -i) &\equiv F(3,16(n+1)+4j+2) + F(3,16(n+1)+4j+1) \\
    &\equiv F(2,8(n+1)+2j+1) \equiv 0 \pmod{2}.
\end{align*}
This proves that $(r(8n+4))_{n\in\N}$ has period $4$ and direct computation shows that $(r(8n+4))_{n\in\N} = \overline{0,1,1,0}$. Therefore, we also have $(r(8n+6))_{n\in\N} = \overline{0,1,1,0}$ so the congruence $r(16n+8) \equiv r(16n+6) + r(8n+4)$ implies $(r(16n+8))_{n\in\N} = \overline{0,0,1,1}$.

It remains to show the last two equalities. For $n=0$ we obviously have $r(32n) = r(16n)$, while for $n \geq 1$ we get
$$ r(32n) \equiv r(32n-2) + r(16n) \equiv r(8(4n-1)+6) + r(16n) \equiv  r(16n) \pmod{2}.$$
Similarly, for all $n \in \N$ we have
$$
r(32n+16) \equiv r(32n+14) + r(16n+8) \equiv  1 + r(16n+8) \pmod{2},
$$
and the proof is complete.
\end{proof}

From these relations, we can derive a DFAO generating the sequence $(r(2n))_{n\in\N}$ (reading the input starting with the least significant digit of $n$). By a simple modification one can also obtain a DFAO that generates the whole sequence $(r(n))_{n\in\N}$.

\begin{figure}[h!]
\centering
\begin{tikzpicture}
\node[state, initial](2n){$1$};
\node[state] at (1.25,1.25) (4n){$1$};
\node[state] at (1.25,-1.25) (4n_2){$0$};
\node[state] at (3,3) (8n){$1$};
\node[state] at (3,-3) (8n_2){$0$};
\node[state] at (2.5,0) (8n_4){$0$};
\node[state] at (3.75,-1.25) (16n_4){$1$};
\node[state] at (3.75,1.25) (16n_12){$0$};
\node[state] at (6.25,3) (16n_8){$0$};
\node[state] at (6.25,-3) (16n_10){$1$};
\node[state] at (9,3) (16n){$1$};
\node[state] at (9,-3) (16n_2){$0$};
\node[state] at (6.25,1.25) (32n_12){$0$};
\node[state] at (6.25,-1.25) (32n_4){$1$};

\draw
(2n) edge[above left] node{$0$} (4n)
(2n) edge[below left] node{$1$} (4n_2)
(4n) edge[above left] node{$0$} (8n)
(4n) edge[above right] node{$1$} (8n_4)
(4n_2) edge[below left] node{$0$} (8n_2)
(4n_2) edge[below right] node{$1$} (8n_4)
(8n) edge[bend left, above] node{$0$} (16n)
(8n) edge[above] node{$1$} (16n_8)
(8n_2) edge[bend right, below] node{$0$} (16n_2)
(8n_2) edge[below] node{$1$} (16n_10)
(8n_4) edge[below left] node{$0$} (16n_4)
(8n_4) edge[above left] node{$1$} (16n_12)
(16n) edge[loop above] node{$0$} (16n)
(16n) edge[bend left=45, left] node{$1$} (16n_10)
(16n_2) edge[loop below] node{$0$} (16n_2)
(16n_2) edge[bend right=45, left] node{$1$} (16n_8)
(16n_4) edge[below] node{$0$} (32n_4)
(16n_4) edge[below right, pos=0.25] node{$1$} (32n_12)
(16n_8) edge[above left] node{$0,1$} (16n_12)
(16n_10) edge[below left] node{$0,1$} (16n_4)
(16n_12) edge[above] node{$0$} (32n_12)
(16n_12) edge[above right, pos=0.25] node{$1$} (32n_4)
(32n_4) edge[loop right] node{$0,1$} (32n_4)
(32n_12) edge[loop right] node{$0,1$} (32n_12);
\end{tikzpicture}
\caption{A DFAO that generates the sequence $(r(2n))_{n\in\N}$.}
\label{r_DFAO}
\end{figure}

\section{Analysis of the sequence ${\bf h}_{1,b}$ with $b\geq 2$}\label{sec4}

The analysis performed in Section \ref{sec3} (more precisely, Theorem \ref{nonregular}) revealed a strong connection of an even indexed subsequence of ${\bf h}$ with sequences counting certain binary partitions. Moreover, we proved that the sequence ${\bf h}$ is neither $k$-regular for any $k$ nor a linear recurrence sequence. We show that a similar phenomenon holds for the sequence ${\bf h}_{1,b}$ with $b\geq 2$. Based on our experience with the sequence ${\bf h}$ and numerical computations, we have observed the following generalization of Theorem \ref{rech}.

\begin{thm}\label{rechb}
For all $n\in\N$ we have the recurrence relations
\begin{align*}
  h_{1,b}(2n)&=n+1, \\
  h_{1,b}(2n+1) &= h_{1,b}(n)+h_{1,b}(2n-1), \quad n \geq 1
\end{align*}
and the identity
$$h_{1,b}(2n+1)= b-1 + \sum_{i=0}^{n}h_{1,b}(i).$$
\end{thm}
\begin{proof}
The proof is by induction on $n$ and is very similar to the proof of Theorem \ref{rech}. More precisely, we first simultaneously prove that $h_{1,b}(2n)=n+1$ and $h_{1,b}(2n-1)\geq 2n$. This is clearly true for $n=0, 1, 2$. Suppose that this is true for all positive integers $\leq n-1$. We then have $$h_{1,b}(2n)=h_{1,b}(2n-h_{1,b}(2n-1))+h_{1,b}(2n-2)=1+n$$
and
$$h_{1,b}(2n+1)=h_{1,b}(2n+1-h_{1,b}(2n))+h_{1,b}(2n-1)=h_{1,b}(n)+h_{1,b}(2n-1).$$ If $n$ is odd then we have
$$
h_{1,b}(2n+1)\geq h_{1,b}\left(2\frac{n+1}{2}-1\right)+2n\geq 2\frac{n+1}{2}+2n=3n+1\geq 2n+2.
$$
If $n$ is even then
$$h_{1,b}(n)+h_{1,b}(2n-1)\geq \frac{n}{2}+1+2n\geq 2n+2$$
and we are done.

To get an expression for $h_{1,b}(2n+1)$ in the form of a sum, we again use induction. The statement is true for $n=0$. Suppose that it holds for $n-1$, where $n \geq 1$. Then
\begin{align*}
h_{1,b}(2n+1)&=h_{1,b}(2n+1-h_{1,b}(2n))+h_{1,b}(2n-1)=h_{1,b}(n)+h_{1,b}(2n-1)\\
             &=h_{1,b}(n)+\sum_{i=0}^{n-1}h_{1,b}(i)+b-1=\sum_{i=0}^{n}h_{1,b}(i)+b-1,
\end{align*}
and the result follows.
\end{proof}

Using the recurrence relation satisfied by the sequence ${\bf h}_{1,b}$ and performing the same type of reasoning as in the proof of Theorem \ref{genfunh} we get the following.

\begin{thm}\label{genfunhb}
Let $b\in\N_{\geq 2}$ and put $H_{b}(x)=\sum_{n=0}^{\infty}h_{1,b}(n)x^{n}$. Then the series $H_{b}$ satisfies the functional equation
$$
H_{b}(x)=\frac{x}{1-x^{2}}H_{b}(x^2)+\frac{(b-1)x}{1-x^2}+\frac{1}{(1-x^2)^{2}}.
$$
\end{thm}

Analogously to Corollary \ref{transH}, we get that $H_b$ is transcendental.

\begin{cor}\label{transHb}
The generating function $H_{b}(x)=\sum_{n=0}^{\infty}h_{1,b}(n)x^{n}$ of the sequence ${\bf h}_{1,b}$ is transcendental over $\Q(x)$.
\end{cor}
 \begin{proof}
The proof goes exactly as in the case of Corollary \ref{transH}. This time, to apply the Nishioka theorem, we note that $\phi_{b}(x, H_{b}(x))=H_b(x^2)$, where
$$
\phi_{b}(x,y)=\frac{1-x^2}{x}y+1-b-\frac{1}{x(1-x^2)}.
$$
As before, it is enough to prove non-rationality of the series $H_{b}$. Again, we consider $G_{b}=H_{b}\pmod*{2}$. Note that $G_{b}$ satisfies $f_{b}(x,G_{b}(x))=0$ (in $\mathbb{F}_{2}[[x]]$) for
$$
f_{b}(x,y)=x(1-x^2)y^2-(1-x^2)^2y+(b-1)x(1-x^2)+1.
$$ However, one can easily check that the polynomials $f_b$ are irreducible in $\mathbb{F}_{2}[x,y]$ and thus $G_{b}$ is not a rational function in $\mathbb{F}_{2}(x)$ for any value of $b\in\N_{+}$. Thus, $H_{b}$ cannot be a rational function over $\Q(x)$.
\end{proof}
%\begin{lem}
%For $n\in\Z$ we have $h_{1,2}(n)=h_{1,1}(n+1)$.
%\end{lem}
%\begin{proof}
%We have $h_{1,2}(0)=h_{1,1}(1)=1, h_{1, 2}(1)=h_{1,1}(2)=2, h_{1,2}(2)=h_{1,1}(3)=2$, so the statement is true for $n=0, 1, 2$.
%\end{proof}
As an application of Corollary \ref{genfunhb} we get the following (a bit) unexpected result, linking the sequences $\mathbf{h}_{1,b}$ for various $b$.

\begin{thm}\label{h1bbin}
For all $n \in \N$ we have
$$ h_{1,2}(n) = h_{1,1}(n+1)  $$
More generally, for $b\in\N_{\geq 2}$ and $n\in\N$ we have the identity
$$
h_{1,b}(2n+1)=(b-2)bin(n+1)+h_{1,1}(2n+2).
$$
\end{thm}
\begin{proof}
The first claim holds for $n=0,1$ by definition, while for $n \geq 2$ it follows from the fact that the sequences $\mathbf{h}_{1,1}$ and $\mathbf{h}_{1,2}$ are defined by the same recurrence.

To obtain the second required equality, we use the generating function approach and first prove that $h_{1,b+1}(2n+1)-h_{1,b}(2n+1)=bin(n+1)$. Let us note that by Corollary \ref{genfunhb}, the series $W(x)=H_{b+1}(x)-H_{b}(x)$ satisfies the functional equation
$$
W(x)=\frac{x}{1-x^2}W(x^2)+\frac{x}{1-x^2}.
$$
Due to the equality $h_{1,b}(2n)=n+1$ we get that $W(x)=-W(-x)$. Next, let us denote $U(x)=\sum_{n=0}^{\infty}bin(n+1)x^{2n+1}$ and consider $B(x)=\sum_{n=0}^{\infty}bin(n)x^{n}$. Note that $B(x^2)=xU(x)+1$ and recall the functional equation for $B$ has the form $(1-x)B(x)=B(x^2)$. Thus, the series $U(x)$ satisfies
the functional equation
$$
U(x)=\frac{x}{1-x^2}U(x^2)+\frac{x}{1-x^2},
$$
which is precisely the equation satisfied by $W(x)$. Because $W(0)=U(0)=1$ we have $W(x)=U(x)$ and hence the coefficients in both series are equal.

To get the general statement, we sum the equalities $h_{1,k+1}(2n+1)-h_{1,k}(2n+1)=bin(n+1)$ from $k=2$ to $k=b-1$, and get
$$
h_{1,b}(2n+1)-h_{1,2}(2n+1)=\sum_{k=2}^{b-1}(h_{1,k+1}(2n+1)-h_{1,k}(2n+1))=(b-2)bin(n+1).
$$
Using the equality $h_{1,2}(2n+1)=h_{1,1}(2n+2)$ proved earlier, we get the result.

\end{proof}

\section{Analysis of the sequence ${\bf h}_{a,b}$ with $a \geq 2$}\label{sec5}

The behavior of the sequence $\mathbf{h}_{a,b}$ in the remaining case $a \geq 2$ is rather trivial, as in each case the subsequences $h_{a,b}(2n)$ and $h_{a,b}(2n+1)$ are eventually arithmetic progressions. In the following, we provide explicit formulas depending on $a,b$.

\begin{thm}  \label{thm:h_a>=2_b>=1}
Assume that $a \geq 2 $ and $b \geq 2$. Then for all $n \in \N$ we have
\begin{align*}
h_{a,b}(2n) &= (n+1)a, \\
h_{a,b}(2n+1) &= na+b.
\end{align*}
\end{thm}
\begin{proof}
We prove both equalities simultaneously by induction on $n$. The case $n=0$ is contained in the definition of $\mathbf{h}_{a,b}$. For $n \geq 1$ we have
$$ h_{a,b}(2n) = h_{a,b}(2n-h_{a,b}(2n-1))+ h_{a,b}(2n-2) = a + na = (n+1)a,$$
where we have used
$$2n-h_{a,b}(2n-1) = 2n - (n-1)a - b \leq 0.$$
Similarly,
$$ h_{a,b}(2n+1) = h_{a,b}(2n+1-h_{a,b}(2n))+ h_{a,b}(2n-1) = a + (n-1)a+b = na+b.  $$
\end{proof}

In the following results, the proofs also rely on easy induction and are omitted.

\begin{thm}  \label{thm:h_a>=3_b=1}
Assume that $a \geq 3$ and $b = 1$. Then $h_{a,1}(0)=a$ and for all $n \in \N_+$ we have
$$ h_{a,1}(n) = \left \lfloor \frac{n}{2} \right\rfloor a+1.  $$
\end{thm}

\begin{thm}  \label{thm:h_a=2_b=1}
We have $h_{2,1}(0)=2, h_{2,1}(1)=1, h_{2,1}(2)=h_{2,1}(3)=3$ and for all $n \in \N_{\geq 2}$ the equalities
\begin{align*}
h_{2,1}(2n) &= 3n-2,\\
h_{2,1}(2n+1) &= 2n.
\end{align*}
\end{thm}

\section{Analysis of the sequence $\mathbf{g}_{a,b}$}\label{sec6}

In this section, we consider the sequence $\mathbf{g}_{a,b} = (g_{a,b}(n))_{n\in\N}$. It turns out that in each case, one of the subsequences $(g_{a,b}(2n))_{n \in\N}$ and $(g_{a,b}(2n+1))_{n \in\N}$ is eventually constant, while the other eventually satisfies a linear recurrence.

We first state a useful observation, which is an immediate consequence of the recursive definition of $g_{a,b}$.
\begin{lem} \label{lem:g_nondecreasing}
The sequences $(g_{a,b}(2n))_{n\in\N}$ and $(g_{a,b}(2n+1))_{n\in\N}$ are nondecreasing.
\end{lem}

In the following series of results we succesively consider various cases $a,b$.

\begin{thm}  \label{thm:g_a>=b>=2}
Assume that either $a=b\geq 2$ are even or $a > b \geq 2$. Then for all $n \in \N$ we have
\begin{align*}
g_{a,b}(2n+1) &= b, \\
g_{a,b}(2n+2) &= g_{a,b}(2n+2-b) + g_{a,b}(2n).
\end{align*}
In particular, we have
$$
g_{a,b}(2n) = \begin{cases}
a &\text{if } n < \lfloor b/2 \rfloor, \\
a + \left(n - \lfloor b/2 \rfloor \right)b &\text{if } b \text{ is odd and } n \geq \lfloor b/2 \rfloor, \\
2^n a &\text{if } b = 2.
\end{cases}
$$
\end{thm}
\begin{proof}
Treating $a,b$ as fixed values, for simplicity, we will write $g = g_{a,b}$.
We will prove by induction on $n$ that $g(2n+1)=b$ and $g(2n) > 2n+1$ for all $n \in \N$, from which the first part of the assertion follows. The case $n=0$ is contained in the definition of $g$. For $n \geq 1$ we have
$$
g(2n+1) = g(2n+1 - g(2n)) + g(2n-1) = b,
$$
where we used the induction hypothesis $g(2n) > 2n+1$.

At the same time, when $2n < b$, we get $a > 2n+1$, which implies
\begin{equation} \label{eq:g(2n)_inequality}
g(2n) \geq g(0) = a > 2n+1.
\end{equation}
When $2n \geq b$, we use the induction hypothesis to obtain
$$ g(2n) = g(2n - g(2n-1)) + g(2n-2) > g(2n - b) + 2n-1. $$
Since $g(2n - b) \geq 2$ by Lemma \ref{lem:g_nondecreasing}, our claim is proved.

The expressions for $g(2n)$ in particular cases quickly follow from the general one.
\end{proof}

\begin{thm} \label{thm:g_b>=a>=2}
Assume that either $b=a \geq 3$ are odd or $b > a \geq 2$. Then for all $n \in \N$ we have
\begin{align*}
g_{a,b}(2n) &= a, \\
g_{a,b}(2n+3) &= g_{a,b}(2n+3-a) + g_{a,b}(2n+1).
\end{align*}
In particular, we have
$$
g_{a,b}(2n+1) = \begin{cases}
b  &\text{if } n < \lfloor a/2 \rfloor, \\
b + \left(n +1 - \lfloor a/2 \rfloor \right)a &\text{if } a \text{ is odd and } n \geq \lfloor a/2 \rfloor, \\
2^n b &\text{if } a = 2.
\end{cases}
$$
\end{thm}
\begin{proof}
The reasoning is based on showing that $g(2n)=a$ and $g(2n+1) > 2n+2$ for all $n \in \N$. This is done in the same way as in the proof of Theorem  \ref{thm:g_a>=b>=2}.
\end{proof}

By comparing the formulas in Theorem \ref{thm:g_a>=b>=2} and \ref{thm:g_b>=a>=2}, we find that for $a > b \geq 2$ the sequences $(g_{a,b}(n))_{n\in\N}$ and $(g_{b,a}(n))_{n\in\N}$ are a shift of one another, where the direction of the shift depends on the parity of $b$. We will, however, prove this corollary in a different way, which exhibits a more intrinsic reason behind this connection.

\begin{cor} \label{cor:g_a,b_shift}
Assume that $a > b \geq 2$. If $b$ is even, then for all $n \in \N$ we have
$$g_{b,a}(n+1) = g_{a,b}(n).$$
If $b$ is odd, then for all $n \in \N$ we have
$$g_{b,a}(n) = g_{a,b}(n+1).$$
\end{cor}
\begin{proof}
We first deal with the case where $b$ is even. We have $g_{b,a}(n+1) = g_{a,b}(n)$ for all integers $n \leq 1$ except for $n = -1$, in which case $g_{b,a}(0)=b$ and $g_{a,b}(-1)=0$. As both considered sequences satisfy the same recurrence relation, it is sufficient to check that the term $g_{a,b}(-1)$ never actually occurs in the computation of $g_{a,b}(n)$ for $n > 1$. In other words, we need to show that
$$ g_{a,b}(n-1) \neq n+1$$
for $n > 1$. If $n = 2m$ is even, this follows from $g_{a,b}(2m-1) = b \equiv 0 \pmod{2}$. If $n=2m+1$ is odd, we use the assumption $a > b$ ($b$ even) to strengthen the inequality \eqref{eq:g(2n)_inequality} to $g_{a,b}(2m) > 2m+2$, and our claim follows.

Along the same lines, reasoning for $b$ odd reduces to showing that $g_{a,b}(n) \neq n+1$ for $n > 1$. The proof in this case is analogous.
\end{proof}

We are left with the case where at least one of $a,b$ is equal to $1$.

\begin{thm} \label{thm:g_a>=3_b=1}
Assume that $a \geq 3$ and $b = 1$. If $a$ is even, then for all $n \in \N$ we have
\begin{align*}
g_{a,1}(2n) &=
\begin{cases}
a+n &\text{if } n < a+1, \\
g_{a,1}(2n-2-a) + g_{a,1}(2n-2) &\text{if } n \geq a+1,
\end{cases} \\
g_{a,1}(2n+1) &=
\begin{cases}
1 &\text{if } n < a-1, \\
a+1 &\text{if } n = a-1, \\
a+2 &\text{if } n > a-1.
\end{cases}
\end{align*}
If $a$ is odd, then for all $n \in \N$ we have
\begin{align*}
g_{a,1}(2n) &=
\begin{cases}
a+n &\text{if } n < a, \\
g_{a,1}(2n-1-a) + g_{a,1}(2n-2) &\text{if } n \geq a,
\end{cases} \\
g_{a,1}(2n+1) &=
\begin{cases}
1 &\text{if } n < a-1, \\
a+1 &\text{if } n \geq  a-1.
\end{cases}
\end{align*}
\end{thm}
\begin{proof}
Since the reasoning in both cases is similar, we show that the assertion holds for $a$ even and leave the verification for $a$ odd to the reader. Again, since $a,b$ are fixed, we omit them in the subscript.
To begin, we simultaneously prove both formulas for $g(2n)$ and $g(2n+1)$ for $n < a-1$ by induction. The base case $n=0$ follows from the definition. For $1 \leq n < a-1$ we have
$$ g(2n) = g(2n-g(2n-1)) + g(2n-2) = g(2n-1) + a+n-1 = a+n,$$
as well as
$$ g(2n+1) = g(2n+1-g(2n)) + g(2n-1) = g(n+1-a) + 1 = 1.$$
For $n=a-1,a$ the values $g(2n)$ and $g(2n+1)$ can be calculated directly.

In the case $n \geq a+1$ we check by induction that $g(2n) > 2n+1$, which immediately implies $g(2n+1) = a+2$. The inequality holds for $n =a+1$, as
$$ g(2a+2) = g(2a+2 - g(2a+1)) + g(2a) = g(a) + 2a = \frac{7}{2} a > 2a + 3, $$
by the assumption $a \geq 3$.
For $n \geq a+2$ we have
$$ g(2n) = g(2n-g(2n-1)) + g(2n-2) > g(2n-2-a)+  2n-1 > 2n+1, $$
using the fact that $(g(2n))_{n\in\N}$ does not decrease. The result follows.
\end{proof}

\begin{thm}
Assume that $a = 1$ and $b \geq 3$. Then for all $n \in \N$ we have
\begin{align*}
g_{1,b}(2n) &=
\begin{cases}
1 &\text{if } n < b-1, \\
2 &\text{if } n \geq b-1,
\end{cases} \\
g_{1,b}(2n+1) &=
\begin{cases}
b+n &\text{if } n < b-1, \\
2^{n-b+3}(b-1) &\text{if } n \geq b-1.
\end{cases}
\end{align*}
\end{thm}
\begin{proof}
The argument is similar to the proof of Theorem \ref{thm:g_a>=3_b=1}. The values of $g(2n)$ and $g(2n+1)$  for $n<b-1$ can easily be derived by induction.

For $n \geq b-1$ we again use induction to prove that $g(2n)=2$ and $g(2n+1)>2n+2$, from which the result follows.
\end{proof}

To finish the investigation of $\mathbf{g}_{a,b}$, we deal with the remaining cases $(a,b) \in \{(1,1),(1,2),(2,1)\}$.
The proof is very similar to the previous ones and is omitted.

\begin{thm}
For all $n \in \N$ we have
\begin{align*}
g_{1,1}(2n+6) &=  g_{1,1}(2n+4)+g_{1,1}(2n+2),  \\
g_{1,1}(2n+7) &= 4,
\end{align*}
and
\begin{align*}
g_{1,2}(2n+6) &= 4,  \\
g_{1,2}(2n+5) &= g_{1,2}(2n+3)+g_{1,2}(2n+1),
\end{align*}
and
\begin{align*}
g_{2,1}(2n+8) &= g_{2,1}(2n+6)+g_{2,1}(2n+2),  \\
g_{2,1}(2n+7) &= 6.
\end{align*}
\end{thm}

\section{Questions and problems}\label{sec7}

In this final section, we formulate some questions and problems related to the results of this paper which may be of interest.

To begin, it is well-known that a linear combination of linear recurrence sequences is again a linear recurrence sequence. It is natural to ask whether the same property holds for meta-Fibonacci sequences (with values in $\N_{+}$).

\begin{ques} \label{linear_combination}
Assume that $\mathbf{f} =(f(n))_{n \in \N}$ is $\Z$-linear combination of meta-Fibonacci sequences such that $f(n) \in \N_+$ for all $n$. Is $\mathbf{f}$ a meta-Fibonacci sequence?
\end{ques}

In the proof of Theorem \ref{h1bbin} we have shown that the sequences $\mathbf{h}_{1,b}$ are closely related to the binary partition sequence by the relation $ h_{1,b+1}(2n+1) - h_{1,b}(2n+1) = bin(n+1)$ for any $b \geq 2$. An affirmative answer to question \ref{linear_combination} would automatically mean that the binary partition sequence is a subsequence of a meta-Fibonacci sequence.

\begin{ques}
Can the sequence $(bin(n))_{n\in\N}$ be realized as a subsequence of a meta-Fibonacci sequence along an arithmetic progression?
\end{ques}

If this is not the case, we may further generalize the class of sequences considered.

\begin{ques}
Can the sequence $(bin(n))_{n\in\N}$ be realized as a subsequence of a recurrence sequence defined by the recurrence
$$ w(n)=\sum_{i=1}^{p}w(n-w(n-u_{i}))+\sum_{i=1}^{q}w(n-v_{i}), $$
with the initial conditions chosen correctly, where $u_{i}, v_{j}$ are positive integers?
\end{ques}

We now move on to the sequence $(r(n))_{n\in\N} = (h_{1,1}(n)\pmod*{2})_{n\in\N}$. We have shown that this sequence is 2-automatic. It is natural to ask whether a similar property holds for the reduction modulo any positive integer.

\begin{ques}
Let $m$ be a positive integer. Is $(h(n) \pmod*{m})_{n\in\N}$ a $2$-automatic sequence?
\end{ques}

We have performed experimental calculations in Mathematica 13 with the help of the \texttt{IntegerSequences} package by Eric Rowland, available at \begin{center}
    \url{https://ericrowland.github.io/packages.html}
\end{center}
The results indicate that the answer is affirmative when $m$ is a power of $2$. On the other hand, when $m$ is not a power of $2$, the sequence $(h(n) \pmod*{m})_{n\in\N}$ is apparently not $k$-automatic for any $k$. In particular $2$-automaticity would imply that the subsequence $(h(2^n) \pmod*{m})_{n\in\N}$ is periodic, which does not appear to be the case (at least for small values of $m$).

%Moreover, from Corollary \ref{genfunhb} we know that the sequence $(h_{1,b}(2n+1)\pmod*{2})_{n\in\N}$ is 2-automatic. In fact, from the equality $h_{1,b}(2n+1)=(b-2)bin(n+1)+h_{1,1}(2n+2)$ and the property $2|bin(n+1)$ for $n\in\N_{+}$ (see, \cite{Ch}) we get that $h_{1,b}(2n+1)\equiv h_{1,1}(2n+2)\pmod*{2}$. Thus, the understanding of the sequence $(r(n))_{n\in\N}$ allow understanding of $(h_{1,b}(2n+1)\pmod*{2})_{n\in\N}$ for all $b$ at once.

Another interesting problem concerns the $k$-regularity of the sequence of indices $n$ such that the value of $r(n) = h_{1,1}(n) \pmod*{2}$  is fixed.

\begin{ques}
Let ${\bf u}_{i}=(u_{i}(n))$ for $i=0, 1$ be the increasing sequence of non-negative integers such that $r(u_{i}(n))=i$. Is the sequence ${\bf u}_{i}$ is $k$-regular for some $k$?
\end{ques}

According to our numerical computations (using the mentioned package)  both sequences ${\bf u}_{0}, {\bf u}_{1}$ are $2$-regular.

Finally, motivated by the results in this paper, one may try to study the behavior of meta-Fibonacci sequences defined by the general recurrence \eqref{genmeta1}. In particular, it seems interesting to consider a direct generalization of the sequence $\mathbf{h}_{a,b}$.

\begin{prob}
Let $u, v$ be fixed positive integers. Let us put $f(n)=a_{0}$ for $n\leq 0$ and for given initial conditions $f(1), \ldots, f(k-1)\in\N_{+}$
investigate the behavior of the sequence $\mathbf{f} =(f(n))_{n\in\N}$, where
$$
f(n)=f(n-f(n-u))+f(n-v)\quad \text{for}\quad n \geq k.
$$
\end{prob}

As discussed in the introduction, the case $u \geq 1$ and $v=1$ is rather easy. Next, if $u=2, v=1$, then without much of effort one can obtain the following characterization of the sequence $\mathbf{f}$ for $k=2$:
\begin{enumerate}
\item $f(n)=n+b+1$ for $a=1, b=1, 2$ and $n\geq 4$;
\item $f(n)=n+2a_1-2$ for $a=1, b\geq 3$ and $n\geq 4$;
\item $f(n)=2n-b$ for $a=2, b=1, 2$ and $n\geq 4$;
\item $f(n)=an+b-a$ for $a\geq 2, b\geq 3$ and $n\geq 4$.
\end{enumerate}
The case $u=1, v=2, k=2, a_{0}, a_{1}\in\N_{+}$, is investigated in the present paper.

For general $u, v,$ experimental results suggest that it is reasonable to split the sequence $\mathbf{f}$ into $v$ subsequences $(f(vn+i))_{n\in\N}$, where $i = 0,1\ldots,v-1$. Apparently, each of these subsequences eventually grows either linearly or exponentially. However, we expect that in some case a new non-trivial phenomena arises.
%In fact, our numerical computations suggest that for $u=1$, fixed $v\geq 2$ and %$(a,b)=(1,1)$, the following is true: there is exactly one $e\in\{1,\ldots, v\}$ such that for each $b\in\{1,\ldots, v\}, b\neq e$, the sequence $f_{1,1}(vn+b)$ is eventually %linear. Moreover, the sequence $f_{1,1}(vn+e)$ is of exponential growth.

For example, in the case $u=1, v=3, k=2$ with $a_{0}=f(1)=1$, numerical computations suggest that  $f(3n+1)=n+29871990902013037527$ for $n\geq 9673$ and $f(3n+2)=n+19162$ for $n\geq 120$. At the same time, assuming that these observations are true, we have that
$$
f(3n)=f(2n-19161)+f(3(n-1))
$$
or equivalently
$$
f(3n)=211+\sum_{i=1}^{n}f(2i-19162)\quad\text{for}\quad n\geq 30.
$$
The growth of the subsequence $(f(3n))_{n\in\N}$ seems to be eventually exponential. A~question arises whether the subsequence $(f(3n))_{n\in\N}$ can be described in terms of known functions.

%\begin{figure}
%       \centering
%         \includegraphics[width=4in]{u_1_v_3.jpg}
%        \caption{The plot of $f_{1,1}(n)$ for $n\leq 400$}
%       \label{u_1_v_3}
%    \end{figure}

%\begin{conj}
%For each $a, b\in\N$ and each $k\in\N$ the sequence
%$$
%\left(\sum_{i=0}^{2^{k-1}-1}t_{i}g_{a,b}(2^{k}n+2^{k-1}-2i)\right)_{n\in\N}
%$$
%is a linear recurrence sequence of degree $b/2$ in case of $a>b$.
%\end{conj}

%\begin{conj}
%For each $a, b\in\N$ and each $k\in\N$ there are linear recurrence sequences $(r_{a,b}(n))_{n\in\N}, (s_{a,b}(n))_{n\in\N}$ such that
%$$
%g_{a,b}(4n)-g(4n-2)=r_{a,b}(n),\quad g_{a,b}(4n+2)-g_{a,b}(4n)=s_{a,b}(n).
%$$
%The degrees of the recurrence sequences are equal to $b/2$ in case of $a>b$.
%\end{conj}

\bigskip

\noindent Bartosz Sobolewski, Jagiellonian University, Faculty of Mathematics and Computer Science, Institute of Mathematics, {\L}ojasiewicza 6, 30 - 348 Krak\'{o}w, Poland\\
e-mail:\;{\tt  bartosz.sobolewski@uj.edu.pl}
\bigskip

\noindent  Maciej Ulas, Jagiellonian University, Faculty of Mathematics and Computer Science, Institute of Mathematics, {\L}ojasiewicza 6, 30 - 348 Krak\'{o}w, Poland\\
e-mail:\;{\tt  maciej.ulas@uj.edu.pl}
\bigskip

 \end{document}